\documentclass[12pt]{article}

\usepackage{authblk}

\usepackage{amsmath}
\usepackage{amsfonts}
\usepackage{amssymb}
\usepackage{amsthm}
\usepackage{mathrsfs}

\usepackage[margin=2cm]{geometry}
\usepackage{graphicx}
\usepackage{xcolor}
\usepackage{enumitem}
\usepackage{hyperref}
\usepackage{tabularx}
\usepackage{float}
\usepackage{bbm}
\usepackage{comment}
\usepackage{dsfont}
\usepackage{setspace}

\usepackage{tikz}
\usetikzlibrary{matrix}
\hypersetup{
	colorlinks   = true,
	linkcolor=blue,
	citecolor=blue
}

\usepackage{tocloft}

\newtheorem{theorem}{Theorem}[section]
\newtheorem{lemma}[theorem]{Lemma}

\newtheorem{corollary}[theorem]{Corollary}

\newtheorem{remark}[theorem]{Remark}

\newcommand{\NN}{\mathbb{N}}

\newcommand{\RR}{\mathbb{R}}

\newcommand{\EE}{\mathbb{E}}
\newcommand{\PP}{\mathbb{P}}
\newcommand{\bbone}{\mathbbm{1}}

\newcommand{\rE}{\mathcal{E}}

\newcommand{\rP}{\mathcal{P}}

\newcommand{\ra}{\rightarrow}

\newcommand{\prt}{\partial}
\newcommand{\vp}{\varphi}

\newcommand{\bX}{\overline{X}}

\newcommand{\pM}{p_{\mathcal{M}}}
\newcommand{\con}{\mathbb{X}}

\DeclareMathOperator{\divv}{div}

\DeclareMathOperator{\id}{id}

\DeclareMathOperator{\law}{Law}
\newcommand{\block}[1]{}

\newcommand{\dcbo}{\rho}
\newcommand{\scbo}{\textbf{M}}
\newcommand{\dmf}{\overline{\rho}}
\newcommand{\dpmf}{\overline{\mu}}
\newcommand{\smf}{\overline{\textbf{M}}}
\newcommand{\diff}{\,\mathrm{d}}

\newcommand{\mc}[1]{\mathcal{#1}}

\newcommand{\norm}[1]{\left\Vert #1\right\Vert}
\newcommand{\inner}[2]{\left\langle #1,\; #2\right\rangle}

\newcommand{\abs}[1]{\left\vert #1\right\vert}

\def\XXint#1#2#3{{\setbox0=\hbox{$#1{#2#3}{\int}$ }
		\vcenter{\hbox{$#2#3$ }}\kern-.6\wd0}}

\author[1]{Hui Huang}
\author[2]{Jethro Warnett}
\affil[1]{Department of Mathematics and Scientific Computing, University of Graz, Universitaetsplatz 3, Graz, 8020, Austria\\
Email: hui.huang1@ucalgary.ca\vspace{.3cm}}
\affil[2]{Mathematical Institute, University of Oxford, Woodstock Road, Oxford, OX2 6GG, United Kingdo,
Email: warnett@maths.ox.ac.uk}

\begin{document}
    \title{Well-posedness and mean-field limit estimate of a consensus-based algorithm for multiplayer games}
    \maketitle

    \textbf{Abstract.} Recently, the paper \cite{chenchene_consensus-based_2023} introduces a derivative-free consensus-based particle method that finds the Nash equilibrium of non-convex multiplayer games, where it proves the global exponential convergence in the sense of mean-field law. 
    This paper aims to address theoretical gaps in \cite{chenchene_consensus-based_2023}, specifically by providing a quantitative estimate of the mean-field limit with respect to the number of particles, as well as establishing the well-posedness of both the finite particle model and the corresponding mean-field dynamics.

    \bigskip


    \bigskip

    
    % \renewcommand{\cftsecleader}{\cftdotfill{\cftdotsep}}	
    % \begingroup
    % \let\clearpage\relax
    % \tableofcontents
    % \allowdisplaybreaks
    
    \section{Introduction}

   \subsection{CBO for Multiplayer games}

    Multiplayer games \cite{narahari_game_2014} are found in various fields, from  advertising \cite{maehara_budget_2015}, to neuroeconomics \cite{king-casas_understanding_2012}, and evolutionary biology \cite{gokhale_evolutionary_2014}. In computer science, they are key to machine learning \cite{nouiehed_solving_2019, deng_local_2021}, federated learning \cite{fan_fault-tolerant_2021}, adversarial learning \cite{li_triple_2017, song_multi-agent_2018}, and reinforcement learning \cite{dai_sbeed_2018, lanctot_unified_2017}. 
In game theory, it is agreed that the optimal strategies of all players are the Nash-equilibrium points \cite{nash_equilibrium_1950}, a point where no player can improve their outcome by changing strategy, assuming that the strategy of all other players is fixed. Several heuristic algorithms have been developed to find the Nash-equilibrium. Examples include the multiobjective particle swarm optimization method \cite{han_adaptive_2018, sedarous_multi-swarm_2018} (based on the particle swarm optimization algorithm \cite{wang_particle_2018}) has been used in conjunction with other algorithms, fictous play \cite{berger_browns_2007} and regret matching \cite{hart_simple_2000}.

    \smallskip

    Recently, a consensus based optimization (CBO) algorithm has been proposed to find the Nash-equilibrium of multiplayer games \cite{chenchene_consensus-based_2023}. The CBO algorithm has been originally proposed in \cite{pinnau_consensus-based_2017, carrillo_analytical_2018}, which is  inspired by collective behavior in nature, such as flocking birds or swarming fish. It is used to solve optimization problems by simulating the interaction of particles that collectively move toward an optimal solution. The agents share information and adjust their positions based on a consensus mechanism, which drives the system toward a global optimum.
    The CBO algorithm boasts many advantages such as being derivative free and being amenable to mathematical analysis, which has now been extended and adapted to address a wide variety of optimization settings. Notable extensions include global optimization on compact manifolds \cite{ha_stochastic_2022}, handling general constraints \cite{beddrich2024constrained,bungert2025mirrorcbo}, and optimizing cost functions with multiple minimizers \cite{bungert_polarized_2022}. Additionally, the CBO framework has been successfully applied to multi-objective problems \cite{borghi_adaptive_2023}, min-max problems \cite{borghi2024particle,huang2024consensus}, and high-dimensional machine learning tasks \cite{carrillo_consensus-based_2021,fornasier_consensus-based_2021}. Further advancements include the integration of momentum \cite{jingrun_chen_consensus-based_2022}, memory effects \cite{huang2025self}, as well as
    jump-diffusion processes \cite{kalise_consensus-based_2023}.
 Rather than attempting to include a necessarily incomplete account of this very fast growing field, we refer to the review paper \cite{totzeck2021trends}, and to \cite{trillos2024cb} for a more recent and relatively comprehensive report.

    % Several papers are able to prove that for the different alterations of the algorithm converge to the global optimum solution in the mean-field limit \cite{ borghi_adaptive_2023, borghi_constrained_2023, bungert_polarized_2022, carrillo_fedcbo_2023, carrillo_consensus-based_2021,carrillo_interacting_2024,huang_consensus-based_2024,fornasier_consensus-based_2021,  jingrun_chen_consensus-based_2022, kalise_consensus-based_2023, riedl_leveraging_2024}.

    \smallskip
    The authors in \cite{chenchene_consensus-based_2023} consider a multi-species CBO algorithm. We are given $M\geq 2$ players, where each player is denoted by $m\in [M]:=\{1,\ldots,M\}$. All players posses a strategy $(x_1,\ldots,x_M)\in\RR^{M\cdot d}$, where we denote the strategy of the $m$th player by $x_m$ and the strategy of the opponents $x_{-m}:=(x_1,\ldots,x_{m-1},x_{m+1}.\ldots, x_M)$. Every player minimizes their own cost function $\rE_m(x_m; x_{-m}):\RR^{M\cdot d}\ra\RR$, that depends on their current strategy and the strategy of the other players. The Nash-equilibrium of the multiplayer game is defined as any strategy $(x_1^\ast,\ldots,x_M^\ast)\in\RR^{M\cdot d}$ that satisfies, for all $m\in [M]$,
    \begin{align*}
        \rE_m(x_m^\ast;\,x_{-m}^\ast)
        \leq \rE_m(x;\,x_{-m}^\ast)
        \qquad\forall x\in \RR^d. 
    \end{align*}
    
    \smallskip
    
     Each player $m$ is attributed a collection of $N$ particles $X^{m,1},\ldots,X^{m,N}\in\RR^d$. The strategy $x_m$ is derived as the consensus of the particles
    \begin{align}\label{eq:consensus}
    	\con_\alpha^m(\dcbo^{m,N}; \scbo^{-m})
    	:=\frac{\int_{\RR^d} x_m \;\omega_\alpha^{\rE_m}(x_m;\scbo^{-m})\,\diff\dcbo^{m,N}(x_m)}{\int_{\RR^d} \omega_\alpha^{\rE_m}(x_m;\scbo^{-m})\,\diff\dcbo^{m,N}(x_m)},
    \end{align}
    where we define weight, particle distribution and player strategies by
    \begin{align*}
    	\omega_\alpha^{\rE_m}(x_m;\scbo^{-m})
    	:=e^{-\alpha \rE_m(x_m;\, \scbo^{-m})},\quad
    	\dcbo^{m,N}:=\frac{1}{N}\sum_{i=1}^N\delta_{ X^{m,i}},\quad
    	\scbo:=\frac{1}{N}\sum_{i=1}^N\left(X^{1,i},\ldots,X^{M,i}\right).
    \end{align*}
    The continuous CBO dynamics is expressed by an interacting particle evolution system represented by the following stochastic differential equation, for all $i\in [N]$ and $m\in [M]$,
    \begin{align}
        \label{eq:cbo}
            \diff X_t^{m,i}
            =-\lambda( X_t^{m,i}-\con_\alpha^m(\dcbo_t^{m,N},\scbo_t^{-m}))\diff t
            +\sigma D( X_t^{m,i}-\con_\alpha^m(\dcbo_t^{m,N},\scbo_t^{-m}))\diff B_t^{m,i},
        \tag{CBO}
    \end{align}
    where $\lambda,\sigma>0$ are drift and diffusion parameters respectively, $\{B^{m,i}\}_{m\in [M], i\in [N]}$ are independent standard Brownian motions and $D:\RR^{d}\ra \RR^{d\times d}$ is Lipshitz transformation with respect to the Frobenius norm. For example $D(X)=\mathrm{diag}(X^1,\ldots,X^d)$ (anisotropic) or $D(X)=|X|\cdot\id$ (isotropic). The paper \cite{chenchene_consensus-based_2023} formally states that the mean-field limit of the above particle dynamics is dictated by the following stochastic differential equation, for all $m\in [M]$,
    \begin{align}
        \label{eq:mf_cbo}
        \diff \bX_t^m
        =-\lambda( \bX_t^{m}-\con_\alpha^m(\dmf_t^m,\smf_t^{-m}))\diff t
        +\sigma D( \bX_t^{m}-\con_\alpha^m(\dmf_t^m,\smf_t^{-m}))\diff B_t^{m,1},
    \tag{MF CBO}
    \end{align}
    where we have the law and expectation of the particles
    \begin{align*}
    	\dmf_t^m := \law\big(\,\bX_t^{m}\big),\qquad
    	\smf_t := \EE\left[\left(\bX_t^1,\ldots,\bX_t^M\right)\right].
    \end{align*}
    Applying It\^{o}’s formula to the joint process $(\bX^1,\ldots,\bX^M)$ shows that the following non-local Fokker-Planck equation holds in the weak sense
        \begin{align*}
            \prt_t \dmf_t
            =\lambda\sum_{m=1}^M\divv_{x_m}\big[(x_m-\con_\alpha^m(\dmf_t^m, \smf_t^{-m}))\; \dmf_t\big]+\frac{\sigma^2}{2}\sum_{m=1}^M\sum_{k=1}^d\prt_{(x_m)_k}^2 \big[(x_m-\con_\alpha^m(\dmf_t^m,\smf_t^{-m}))_k^2\; \dmf_t\big],
        \end{align*}
        where $\dmf$ is the law of the joint distribution. For the analysis of the CBO-type PDE, we refer the readers to \cite{wang2025mathematical}.

\subsection{Motivation}
    The authors \cite{chenchene_consensus-based_2023} rigorously prove global convergence in the large particle limit. Specifically, they establish the convergence of the mean-field dynamics $(\overline{X}_t^1, \dots, \overline{X}_t^M)$ in \eqref{eq:mf_cbo} to the global Nash equilibrium (NE) point $(x_1^*, \dots, x_M^*)$ as $t$ approaches infinity. To achieve this, they introduce the variance functions:
\begin{equation}\label{eq:defi_V_m}
    V^m(t) = \mathbb{E}\left[|\overline{X}_t^m - x_m^*|^2\right], \quad \text{for each} \ m \in [M], \quad \text{and} \quad V(t) := \sum_{m=1}^M V^m(t), \quad \text{for} \ t > 0.
\end{equation}
By analyzing the decay behavior of the (cumulative) variance function $V(t)$, they demonstrate that $V(t)$ decays exponentially at a finite time $T_* > 0$, with a decay rate that can be controlled through the parameters of the CBO method. Specifically, it holds that
\[
    V(t) \leq V(0) \exp\left(-\frac{2\lambda - \sigma^2}{2} t\right)
\]
for $t \in [0, T_*]$, and $V(T_*) \leq \varepsilon$ for any given accuracy $\varepsilon > 0$.

    However, the justification for passing to the mean-field limit from the particle system \eqref{eq:cbo} to the mean-field dynamics \eqref{eq:mf_cbo} remains only formal, and a rigorous mathematical analysis is  absent in \cite{chenchene_consensus-based_2023}. Establishing a rigorous proof of the CBO models has been a challenging task, primarily due to the fact that the consensus point defined in \eqref{eq:consensus} is only locally Lipschitz. Several significant results regarding the mean-field limit for CBO have been achieved in recent years. For instance, \cite{fornasier_consensus-based_2020,ha_stochastic_2022} established mean-field limit estimates for variants of CBO constrained to compact manifolds by ensuring the consensus point is globally Lipschitz. Subsequently, \cite{huang_mean-field_2022} proved the mean-field limit for the standard CBO model using a compactness argument based on Prokhorov's theorem. However, this approach does not provide an explicit convergence rate in terms of the number of particles $N$. 
Further advancements were made in \cite{fornasier_consensus-based_2024,huang2023global}, where the authors demonstrated a probabilistic mean-field approximation, showing that the mean-field limit estimate holds with high probability. The high-probability assumption was later removed by \cite{gerber_mean-field_2024}, who established an improved stability estimate for the consensus point. Most recently, uniform-in-time type of mean-field estimates have been obtained \cite{huang2024uniform,gerber2025uniform,bayraktar2025uniform}.

    The primary objective of this paper is to address some theoretical gaps in \cite{chenchene_consensus-based_2023} by providing a quantitative estimate of the mean-field limit with respect to the number $N$ of particles. Additionally, we establish the well-posedness of both the finite particle model \eqref{eq:cbo} and the mean-field dynamics \eqref{eq:mf_cbo}.

    % Moreover, various papers have also provided a rigorous proof for the mean-field limit \cite{carrillo_interacting_2024, fornasier_consensus-based_2020,fornasier_consensus-based_2024, fornasier_consensus-based_2021,ha_stochastic_2022, huang_uniform--time_2024, huang_mean-field_2022,  kalise_consensus-based_2023}.

    \subsection{Main Results}
    Throughout this paper we assume that the cost functions are sandwiched in between two polynomials and that the Lipshitz constants grow polynomially.
    \begin{enumerate}[label=(A\arabic*)]
        \item \label{ass:lip} There exists constants $C>0$ and $s\geq 0$ such that for all $(x_0,y_0),(x_1,y_1)\in\RR^d \times \RR^{(M-1)\cdot d}$ and $m\in [M]$ we have the local Lipshitz property
    		\begin{align*}
    			|\rE_m(x_0;y_0)-\rE_m(x_1;y_1)|
    			\leq C(1+|(x_0,y_0)|+|(x_1,y_1)|)^s\cdot |(x_0-x_1,y_0-y_1)|.
    		\end{align*}
        \item \label{ass:bnd} There exists constants $0\leq \ell$ and $c,G>0$ such that for all $m\in [M]$ we have the growth
    		\begin{align*}
    			\frac{1}{c}\,\big(|(x,y)|^\ell-G\big)
    			\leq \rE_m(x;y)
    			\leq c\, \big(|(x,y)|^\ell+G\big).
    		\end{align*}
    \end{enumerate}

    For the sake of simplicity we define the variable 
    \begin{align*}
        \pM := \begin{cases}
            2+s & \text{ if }\ell=0,\\
            1 & \text{ if }\ell>0.
        \end{cases}
    \end{align*}

    We adapt two existence results directly from \cite[Theorem 2.2, Theorem 2.3]{gerber_mean-field_2024} for the multi-species setting.

    \begin{theorem}[{Existence and uniqueness for \eqref{eq:cbo}}]
    	\label{theorem:existence_and_uniqueness_for_cbo}
        Let Assumptions \ref{ass:lip} and \ref{ass:bnd} hold. Then the SDEs \eqref{eq:cbo} posses unique strong solutions $\{X_t^{m,i}\}_{m\in [M],i\in [N]}$ for any initial conditions $\{X_0^{m,i}\}_{m\in [M],i\in [N]}$ that are independent of the Brownian motions $\{B_t^{m,i}\}_{m\in [M],i\in [N]}$. The solutions are almost surely continuous.
    \end{theorem}
    
    \begin{theorem}[{Existence and uniqueness for \eqref{eq:mf_cbo}}]
    	\label{theorem:existence_and_uniqueness_for_mfcbo}
        Let Assumptions \ref{ass:lip} and \ref{ass:bnd} hold and $p\geq 2\vee \pM$. Then, for all $T > 0$ and $\dmf_0^1,\ldots,\dmf_0^M\in\rP_p(\RR^d)$, there exists unique processes $\{\bX^m:\Omega\ra C^0([0,T], \RR^d)\}_{m\in [M]}$ satisfying \eqref{eq:mf_cbo} in the strong sense with initial condition $\law(\bX_0^m)= \dmf_0^m$. Furthermore, we have the bounds
        \begin{align}\label{eq:mbounds}
            \begin{split}
                \sup_{m\in[M]}\EE \left[ 
                \sup_{\substack{t \in [0, T]}} 
                \;|\bX_t^m|^p 
                \right] < \infty, \quad & \raisebox{.4cm}{$\displaystyle\sup_{\substack{t \in [0, T] \\ m\in [M]}}$}\; |\con_\alpha^m(\dmf_t^m,\smf_t^{-m})| < \infty,\\
                &\quad\dmf_t^m = \law\big(\,\bX_t^m\big),\quad
                \smf_t= \EE\big[(\bX_t^1,\ldots,\bX_t^M)\big]
            \end{split}
        \end{align}
        and the function $t \mapsto \con_\alpha^m(\dmf_t^m,\smf_t^{-m})$ is continuous over $[0, T]$.
    \end{theorem}

% RELIC
\block{
        Finally, for the joint distribution $\dmf_t=\law(\bX_t^1,\ldots,\bX_t^M)$, the function $t\mapsto \dmf_t$ belongs to $C^0([0,T],\rP_p(\RR^{M\cdot d}))$ and satisfies the following non-local Fokker-Planck equation in the weak sense
        \begin{align*}
            \prt_t \dmf_t
            =\lambda\sum_{m=1}^M\divv_{x_m}\big[(x_m-\con_\alpha^m(\dmf_t^m, \smf_t^{-m}))\; \dmf_t\big]+\frac{\sigma^2}{2}\sum_{m=1}^M\sum_{k=1}^d\prt_{(x_m)_k}^2 \big[(x_m-\con_\alpha^m(\dmf_t^m,\smf_t^{-m}))_k^2\; \dmf_t\big].
        \end{align*}}

    We also close the gap in \cite{gerber_mean-field_2024}, by showing a quantitative mean-field limit estimate.
    
    \begin{theorem}[{Mean-field limit of \eqref{eq:cbo}}]
        \label{theorem:meanfield_of_cbo}
    	Let Assumptions \ref{ass:lip} and \ref{ass:bnd} hold with $q\geq 4\vee 2\pM$, $p\in (0,\frac{q}{2}]$ and $\{\rho_0^m\}_{m\in [M]}\subseteq\rP_q(\RR^d)$. We assume the particles $\{X^{m,i}\}_{m\in [M],i\in[N]}$ satisfy \eqref{eq:cbo} and $\{\bX^{m,i}\}_{m\in [M],i\in[N]}$ are $N$ i.i.d. samples for each player from \eqref{eq:mf_cbo}. They both use the same standard Brownian motions $\{B^{m,i}\}_{m\in [M],i\in[N]}$, with the same initial condition $\law(X_0^{m,i})=\law(\bX_0^{m,i})=\rho_0^m$. Then for each time $T>0$, there exists a positive constant $C>0$ independent of $N$ such that
        \begin{align*}
        \raisebox{.35cm}{$\displaystyle\sup_{\substack{m\in [M]\\i\in [N] }}$}
    		\;\left(\EE\left[\sup_{t\in [0,T]}\;\abs{X_t^{m,i}- \bX_t^{m,i}}^p\right]\right)^{\frac{1}{p}}\leq C N^{-\gamma},
    	\end{align*}
        where we define the exponent
        \begin{align*}
            \gamma:=\min\left\{\frac{1}{2},\; \frac{q-p}{2p^2},\; \frac{q-(2\vee \pM)}{2 (2\vee \pM)^2}\right\}.
        \end{align*}
    \end{theorem}

    \begin{remark}
        We achieve Monte-Carlo convergence rate, this means $\gamma=\frac{1}{2}$ and $p=2$, whenever the following lower bound is satisfied $q\geq 6\vee \big((2\vee \pM)+(2\vee \pM)^2\big)$.
    \end{remark}

    \subsection{Notation}

    In this section we ferment the notation that we use throughout the paper. For any positive values $1<p<\infty$ and $R>0$ we let $\rP_{p,R}(\RR^d)$ (or $\rP_p(\RR^d)$) denote the space of probability measures with $p$-moment bounded by $R$ (or finite $p$-moment), that means
    \begin{align*}
        \int_{\RR^d}|x|^p\diff\mu(x)\leq R,\qquad
        \int_{\RR^d}|x|^p\diff\nu(x)<\infty,
        \qquad\forall \mu\in \rP_{p,R}(\RR^d) \quad \forall \nu\in \rP_p(\RR^d).
    \end{align*}
    Additionally, we define the exptected value of every probabiliy measure $\mu\in\rP_p(\RR^d)$ by
    \begin{align*}
        \EE[\mu]:=\int_{\RR^d} x\diff \mu(x).
    \end{align*}
    Finally, for any vector $x\in\RR^d$ we let $\delta_x$ denote the Dirac measure at $x$.
    
    \subsection{Paper structure}
        \label{section:particle_stab_weighted_mean}
        The paper will be subdivided into four sections. First in section \ref{section:necessary_lemmas} we state several necessary lemmas for the paper. Then, in section \ref{section:wellposedness_for_cbo}, \ref{section:wellposedness_for_mfcbo} and \ref{section:meanfield_limit_for_cbo} we prove the Theorems \ref{theorem:existence_and_uniqueness_for_cbo}, \ref{theorem:existence_and_uniqueness_for_mfcbo} and \ref{theorem:meanfield_of_cbo} respectively.

    \section{Necessary Lemmas}
    \label{section:necessary_lemmas}

    In this section we adapt several results from \cite{gerber_mean-field_2024, carrillo_analytical_2018} to the multi-species setting. We will use these to prove Theorems \ref{theorem:existence_and_uniqueness_for_cbo}, \ref{theorem:existence_and_uniqueness_for_mfcbo} and \ref{theorem:meanfield_of_cbo}. We state a Wasserstein stability result of the consensus in section \ref{subsection:wasserstein:stability_estimate_for_weighted_mean}, we derive an upper bound of the consensus in section \ref{section:bound_on_weighted_moments}, we control the moment of the solution to \eqref{eq:cbo} in section \ref{section:moment_estimate_for_cbo_dynamics} and finally we derive a mean-field limit result in section \ref{section:meanfield_limit_of_consensus_for_mfcbo}.

    \subsection{Wasserstein stability estimate for weighted mean}
    \label{subsection:wasserstein:stability_estimate_for_weighted_mean}
    
    \begin{lemma}
    	\label{lem:wasserstein_stability_estimate_for_weighted_mean}
    	Suppose that Assumptions \ref{ass:lip} and \ref{ass:bnd} hold. Then for all $R>0$ and $p\geq p_{\mc M}$, there exists a constant $C>0$ depending on $R$, $M$ and $p$ such that for all $\mu_1,\ldots,\mu_M\in\rP_{p,R}(\RR^d)$, $\nu_1,\ldots,\nu_M\in \rP_p(\RR^d )$ we have, for all $m\in[M]$,
    	\begin{align*}
    		|\con_\alpha^m(\mu_m, \boldsymbol{\bar \mu}^{-m})-\con_\alpha^m(\nu_m,\boldsymbol{\bar \nu}^{-m})|
    		\leq C \sum_{j=1}^M W_p(\mu_j,\nu_j),
    	\end{align*}
    	where we define the expectations
        \begin{align*}
            \boldsymbol{\bar \mu}^{-m}&:=(\EE[\mu_1],\ldots,\EE[\mu_{m-1}],\EE[\mu_{m+1}]\ldots,\EE[\mu_M]),\\
            \boldsymbol{\bar \nu}^{-m}&:=(\EE[\nu_1],\ldots,\EE[\nu_{m-1}],\EE[\nu_{m+1}]\ldots,\EE[\nu_M]).
        \end{align*} 
    \end{lemma}
    \begin{proof}
    Let us define the probability measure $\bar{\mu}_m:=\mu_m\otimes\delta_{\boldsymbol{\bar \mu}^{-m}}$ and $\bar{\nu}_m:=\nu_m\otimes\delta_{\boldsymbol{\bar \nu}^{-m}}$. We compute the $p$-mean of $\bar{\mu}_m$ by using the Jensen inequality
    \begin{align*}
        \int_{\RR^{M\cdot d}} |y|^p \diff \bar{\mu}_m(y)
        \leq M^{p-1}\left(\int_{\RR^d} |x|^p \diff \mu_m(x)+\sum_{j\neq m}\EE[\mu_j]^p\right)
        \leq M^p R. 
    \end{align*}
    
    Then from \cite[Corollary 3.3]{gerber_mean-field_2024} we know there exists a constant $C>0$ depending only on $R$ and $p$ such that
    	\begin{align*}
    		|\con_\alpha^m(\mu_m,\boldsymbol{\bar \mu}^{-m})-\con_\alpha^m(\nu_m,\boldsymbol{\bar \nu}^{-m})|
    		&\leq C W_p(\bar \mu_m,\bar \nu_m)\\
    		&\leq C M^{\frac{p-1}{p}} W_p(\mu_m,\nu_m)+CM^{\frac{p-1}{p}}\sum_{j\neq m}\big(|\EE[\mu_j]-\EE[\nu_j]|^p\big)^{\frac{1}{p}}.
    	\end{align*}
    	Let $\pi_j\in \Gamma(\mu_j,\nu_j)$ be an arbitrary coupling of $\mu_j$ and $\nu_j$. Note by the Jensen inequality
    	\begin{align*}
    		|\EE[\mu_j]-\EE[\nu_j]|
    		\leq \left(\int_{\RR^d\times\RR^d}|x-y|^p\diff \pi_j(x,y)\right)^{\frac{1}{p}}\,.
    	\end{align*}
  Then we optimize over all couplings $\pi_j$ and have
  \begin{align*}
      |\EE[\mu_j]-\EE[\nu_j]|
    		\leq W_p(\mu_j,\nu_j).
  \end{align*}
    	Thus combining both inequalities delivers the result.
    \end{proof}

    The following corollary is analogue to \cite[Lemma 2.1]{carrillo_analytical_2018}, but we provide a different proof.
    
    \begin{corollary}[local Lipshitz property]
        \label{cor:loc_lip_prop}
        We define the function $F_N^{m}:\RR^{M\cdot N\cdot d}\ra \RR^d$ by 
        \begin{align*}
            F_N^m(X^1,\ldots,X^M)
            &:=\sum_{i=1}^N \frac{\omega_\alpha^{\rE_m}(X^{m,i},\scbo^{-m})}{\sum_{j=1}^N \omega_\alpha^{\rE_m}(X^{m,j},\scbo^{-m})}\cdot X^{m,i},
        \end{align*}
        where $\{X^{m,i}\}_{m\in [M],i\in[N]}\subseteq \RR^{ d}$, and we define $X^m:=(X^{m,1},\ldots,X^{m,N})$ and 
        \begin{align*}
            \scbo^{-m} :=\frac{1}{N}\sum_{i=1}^N \big(X^{1,i},\ldots,X^{m-1,i},X^{m+1,i},\ldots,X^{M,i}\big).
        \end{align*}
        Then for all $R>0$, there exists a constant $C>0$ depending on $R$ and $N$ such that, for all $X_0,X_1\in \RR^{M\cdot N\cdot d}$ with $|X_0|\vee |X_1|\leq R$,
        \begin{align}
            \label{eq:loc_lip_cons}
            |F_N^m(X_0)-F_N^m(X_1)|\leq C |X_0-X_1|.
        \end{align}
    \end{corollary}
    \begin{proof}
        Let $R>0$ and $X_0,X_1\in \RR^{M\cdot N\cdot d}$ with $|X_0|, |X_1|\leq R$. Then, define the particle distributions and expectations by
        \begin{align*}
            \begin{split}
                \dcbo_0^m&:= \frac{1}{N}\sum_{i=1}^N \delta_{X_0^{m,i}},\\
                \scbo_0&:=\big(\EE[\rho_0^1],\ldots,\EE[\rho_0^M]\big),
            \end{split}
            \begin{split}
                \dcbo_1^m&:= \frac{1}{N}\sum_{i=1}^N \delta_{X_1^{m,i}},\\
                \scbo_1&:=\big(\EE[\rho_1^1],\ldots,\EE[\rho_1^M]\big).
            \end{split}
        \end{align*}
        Note that we have the identity, for $k\in\{0,1\}$,
        \begin{align*}
            \con_\alpha^m(\dcbo_k^m, \scbo_k^{-m})
            =F_N^m(X_k^1,\ldots,X_k^M).
        \end{align*}
        Then by Lemma \ref{lem:wasserstein_stability_estimate_for_weighted_mean} for any $p\geq \pM$ fixed, there exists a constant $C>0$ depending on $R$ and $p$ such that
        \begin{align*}
    		|F_N^m(X_0^1,\ldots,X_0^M)-&F_N^m(X_1^1,\ldots,X_1^M)|
            =|\con_\alpha^m(\dcbo_0, \scbo_0^{-m})-\con_\alpha^m(\dcbo_1,\scbo_1^{-m})|\\
    		&\leq C \sum_{m=1}^N W_p(\dcbo_0^m,\dcbo_1^m)
            \leq C \sum_{m=1}^M\left(\sum_{i=1}^N |X_0^{m,i}-X_1^{m,i}|^p\right)^{\frac{1}{p}}\\
            &\leq C \sum_{m=1}^M\sum_{i=1}^N |X_0^{m,i}-X_1^{m,i}|
            \leq C\sqrt{MN} |X_0-X_1|.
    	\end{align*}
        Hence, we have proven the local Lipshitz property as required.
    \end{proof}

    \subsection{Bound on weighted moments}
    \label{section:bound_on_weighted_moments}
    We adapt the following lemma from \cite[Proposition A.3]{gerber_mean-field_2024}.
\begin{lemma}\label{lemma:est_mean_cost}
	Let $p\geq 1$, $\dcbo_1,\ldots,\dcbo_M\in \rP_p(\RR^d)$ and let Assumption \ref{ass:bnd} hold. Then there exists constants $C> 0$ depending on $p,\ell,c, G$ and $M$ such that for all $m\in [M]$, it holds that
	\begin{equation*}
		|\con_\alpha^m(\dcbo_{m},\overline \scbo^{-m})|\leq \frac{\displaystyle\int_{\RR^{d}} |x|\, e^{-\alpha\rE_m(x;\overline\scbo^{-m})}\diff \dcbo_m(x)}{\displaystyle\int_{\RR^{d}} e^{-\alpha\rE_m(x;\overline\scbo^{-m})}\diff \dcbo_m(x)}\leq C\left(\sum_{j=1}^M\int_{\RR^d}|x|^p\diff \dcbo_j(x)\right)^\frac{1 }{p },
	\end{equation*}
    where $\overline\scbo=(\EE[\dcbo_1],\ldots,\EE[\dcbo_M])$.
\end{lemma}
\begin{proof}
We define the probability measure $\dcbo:=\dcbo_m\otimes \delta_{\overline \scbo^{-m}}$. Then from \cite[Proposition A.3]{gerber_mean-field_2024}, there exists a constant $C> 0$ depending on $p$, $\ell$, $c$ and $G$ such that, for all $m\in [M]$,
    \begin{align*}
        |\con_\alpha^m(\dcbo_m,\smf^{-m})|&\leq \frac{\displaystyle\int_{\RR^{d}} |x|\, e^{-\alpha\rE_m(x;\overline \scbo^{-m})}\diff \dcbo_m(x)}{\displaystyle\int_{\RR^{d}} e^{-\alpha\rE_m(x;\overline \scbo^{-m})}\diff \dcbo_m(x)}\\
        &\leq \frac{\displaystyle\int_{\RR^{M\cdot d}} |y|\, e^{-\alpha\rE_m(y)}\diff \dcbo(y)}{\displaystyle\int_{\RR^{M\cdot d}} e^{-\alpha\rE_m(y)}\diff \dcbo(y)}
        \leq C M^{\frac{p-1}{p}}\left(\int_{\RR^d}|x|^p\diff \dcbo_m(x)+\sum_{j\neq m} |\EE[\dcbo_j]|^p\right)^\frac{1}{p}.
    \end{align*}
    If we use the Jensen inequality, we get the desired bound 
    \begin{align*}
        \left(\int_{\RR^d}|x|^p\diff \dcbo_m(x)+\sum_{j\neq m} |\EE[\dcbo_j]|^p\right)^\frac{1}{p}
        \leq \left(\sum_{j=1}^M\int_{\RR^d}|x|^p\diff \dcbo_j(x)\right)^\frac{1}{p}.
    \end{align*}
    We conclude the proof by recombining the above terms.
\end{proof}
\begin{remark}\label{rmk}
    The above estimate still holds  when $\rho_m=\rho^{m,N}=\frac{1}{N}\sum_{i=1}^N\delta_{X^{m,i}}$ is an empirical measure and $\EE[\rho^{m,N}]=\frac{1}{N}\sum_{i=1}^N X^{m,i}$ is the sample average.
\end{remark}

    \subsection{Moment estimate for the CBO particle system}
    \label{section:moment_estimate_for_cbo_dynamics}

    We control the moments of the particles in \eqref{eq:cbo} using the same strategy as in \cite[Lemma 3.5]{gerber_mean-field_2024}.
     
    \begin{lemma}[Moment estimates for the empirical measures]
    	\label{lemma:moment_estimates_for_the_empirical_measures}
    	Let $p\geq 2$, $\bar\rho_0^1,\ldots,\bar\rho_0^M\in \rP_p(\RR^d)$ and let Assumptions \ref{ass:lip} and \ref{ass:bnd} hold. Let $\{X^{i,m}\}_{i\in [N], m\in [M]}$ be solutions to the SDEs \eqref{eq:cbo}, where $\law(X_0^{m,1},\ldots,X_0^{m,N})=(\bar \rho_0^m)^{\otimes N}$. Then there exists a constant $\kappa>0$ that does not depend on $N$ such that
    	\begin{align*}
    		\sup_{\substack{m\in [M] \\ i \in [N] }}\EE \left[ \sup_{t \in [0, T]} 
    		\;|X_t^{m,i}|^p 
    		\right]\leq \kappa\cdot \sum_{m=1}^M\EE[|X_0^{m,1}|^p].
    	\end{align*}
    \end{lemma}
    \begin{proof}
    	We use Lemma \ref{lemma:est_mean_cost} and Remark \ref{rmk}, and know that there exists a constant $C>0$ depending on $p$, $\ell$, $c$, $G$ and $M$ such that 
    	\begin{align*}
    		\con_\alpha^m(\dcbo_t^{m,N},\scbo_t^{-m})
      \leq C\left(\sum_{j=1}^M \int_{\RR^d}|x|^p\diff \dcbo_t^{j,N}(x)\right)^{\frac{1}{p}},
    	\end{align*}
    Thus, there exists a constant $C>0$ independent of $N$ such that we can estimate the drift and diffusion terms by
    	\begin{align*}
    		|X_t^{m,i}-\con_\alpha^m(\dcbo_t^{m,N}, \scbo_t^{-m})|
    		\vee |D(X_t^{m,i}-\con_\alpha^m(\dcbo_t^{m,N}, \scbo_t^{-m}))|
    		\leq C\left(|X_t^{m,i}|^p+\sum_{j=1}^M\int_{\RR^d}|x|^p\diff\dcbo_t^{j,N}(x)\right)^{\frac{1}{p}}.
    	\end{align*}
    	Next we apply the Burkholder-Davis-Grundy inequality \cite[Chapter 1, Theorem 7.3]{mao_stochastic_2011} to find for all $t\in [0,T]$ that
    	\begin{multline*}
    		\frac{1}{3^{p-1}}\EE\left[\sup_{s\in [0,t]}\abs{X_t^{m,i}}^p\right]
    		\leq \EE[|X_0^{m,i}|^p]+\lambda^p T^{p-1}\int_0^t \EE\left[\abs{X_s^{m,i}-\con_\alpha^m(\dcbo_s^{m,N},\scbo_s^{-m})}^p\right]\,\diff s\\
    		+T^{\frac{p}{2}-1}\sigma^p C_{\mathrm{DBG}} \int_0^t \EE\left[\abs{D(X_s^{m,i}-\con_\alpha^m(\dcbo_s^{m,N},\scbo_s^{-m}))}^p\right]\,\diff s\\
    		\leq \EE[|X_0^{m,i}|^p]+C\left(\int_0^t \EE\left[|X_s^{m,i}|^p+\sum_{j=1}^M\int_{\RR^d}|x|^p\diff\dcbo_s^{j,N}(x)\right]\diff s\right).
    	\end{multline*}
    	As the particles $X^{m,1},\ldots,X^{m,N}$ are exchangeable, we see that
    	\begin{align*}
    		\EE\left[\int_{\RR^d}|x|^p\diff\dcbo_t^{m,N}(x)\right]=\EE[|X_t^{m,i}|^p].
    	\end{align*}
    	Hence, for some constant $C>0$, we derive
    	\begin{align*}
    		\EE\left[\sup_{s\in [0,t]}\abs{X_s^{m,i}}^p\right]
    		\leq C\left(\EE[|X_0^{m,i}|^p]+\int_0^t \sum_{j=1}^M\EE\left[\sup_{r\in [0,s]}\abs{X_r^{j,i}}^p\right]\diff s\right).
    	\end{align*}
    	We sum over $m=1,\ldots,M$ and get the inequality
    	\begin{align*}
    		\sum_{m=1}^M\EE\left[\sup_{s\in [0,t]}\abs{X_s^{m,i}}^p\right]
    		\leq C\left( \sum_{m=1}^M\EE[|X_0^{m,i}|^p]+\int_0^t \sum_{m=1}^M\EE\left[\sup_{r\in [0,s]}\abs{X_r^{m,i}}^p\right]\diff s\right).
    	\end{align*}
    	The result now follows from the Grönwall inequality.
    \end{proof}

    \begin{remark}
        \label{remark:add_xi_to_cbo}
        If we exchange the SDE \eqref{eq:cbo} with the one below
        \begin{align*}
            \diff X_t^{m,i}
            =-\lambda( X_t^{m,i}-\xi \con_\alpha^m(\dcbo_t^{m,N},\scbo_t^{-m}))\diff t
            +\sigma D( X_t^{m,i}-\xi \con_\alpha^m(\dcbo_t^{m,N},\scbo_t^{-m}))\diff B_t^{m,i}
        \end{align*}
        for any $\xi\in [0,1]$, then the Lemma \ref{lemma:moment_estimates_for_the_empirical_measures} still holds. The proof of this is analogue.
    \end{remark}

    \subsection{Convergence of the weighted mean for i.i.d. samples}
    \label{section:meanfield_limit_of_consensus_for_mfcbo}
    In this section we adapt the proof of \cite[Lemma 3.7]{gerber_mean-field_2024}. The authors directly apply the result in \cite[Theorem 1]{doukhan_evaluation_2009} to derive a mean-field limit. The only thing we need to adapt in the proof of \cite[Lemma 3.7]{gerber_mean-field_2024} is to explicitly control a constant.

    \begin{theorem}[{\cite[Theorem 1]{doukhan_evaluation_2009}}]
        \label{thm:mf_stationary_sequence}
        Let $\{w_j, V_j\}_{j \in \mathbb{N}}$ be a stationary sequence with values in $\mathbb{R} \times \mathbb{R}^d$ with $w_j > 0$ almost surely, and set
        \begin{align*}
            \hat{N}_J = \frac{1}{J} \sum_{j=1}^{J} w_j V_j, \qquad
            \hat{D}_J = \frac{1}{J} \sum_{j=1}^{J} w_j, \qquad
            \hat{R}_J = \frac{\hat{N}_J}{\hat{D}_J}.
        \end{align*}
        Let also $N = \EE[\hat{N}_1]$, $D = \EE[\hat{D}_1]$, and $R = \frac{N}{D}$. Let $0 < p < q$ and assume that for some $\theta,C > 0$,
        \begin{align}
            \label{eq:uniform_bnds_for_stationary_sequence}
            \begin{split}
                &r:=\frac{p(q+2)}{q-p},\quad
                s:=\frac{pq}{q-p},\quad
                \EE[w_1^q]\leq \theta,\quad
                \EE[|V_1|^r]\leq \theta, \quad
                \EE[|w_1 V_1|^s]\leq \theta,\\
                &\qquad\qquad\quad\EE[|\hat{D}_J-D|^q]^{\frac{1}{q}}
                \leq C J^{-\frac{1}{2}},\qquad
                \EE[|\hat{N}_J-N|^p]^{\frac{1}{p}}
                \leq C J^{-\frac{1}{2}}.
            \end{split}
        \end{align}
        Then the following inequality is satisfied 
        \begin{align}
            \label{eq:mf_stationary_sequence}
            \EE[|\hat{R}_J - R|^p] \leq  \frac{C}{D}\left(1+2\frac{N}{D}+\frac{\theta}{D}+\theta\left(\frac{C}{D}\right)^{2/r}\right)\; J^{-\frac{1}{2}}.
        \end{align}
    
    \end{theorem}
    
    \begin{lemma}[Convergence of the weighted mean for i.i.d. samples]
    	\label{lemma:convergence_of_the_weighted_mean_for_iid_samples}
         Let Assumptions \ref{ass:lip} and \ref{ass:bnd} hold and let $2\leq p<r$ and $R>0$. Then for all $\mu\in\rP_r(\RR^d)$ there exists constant $C>0$ depending on $\mu$, $R$, $p$ and $r$ such that for all $Y\in \RR^{(M-1)\cdot d}$ with $|Y|\leq R$ and $N\in \NN$ we have
    	 \begin{align*}
    	 	\EE[|\con_\alpha^m(\mu_N,Y)-\con_\alpha^m(\mu,Y)|^p]
    	 	\leq C N^{-\frac{p}{2}},\qquad 
    	 	\mu_N=\frac{1}{N}\sum_{j=1}^N \delta_{X^j},\qquad
    	 	\{X^j\}_{j\in\NN}\overset{\text{i.i.d.}}{\sim} \mu.
    	 \end{align*}
    \end{lemma}
    \begin{proof}
        We adapt the proof of \cite[Lemma 3.7]{gerber_mean-field_2024}. In particular, we apply Theorem \ref{thm:mf_stationary_sequence}. We show that the coefficient in \eqref{eq:mf_stationary_sequence} is uniformly bounded from above for all $|Y|\leq R$. To that end, we define the i.i.d. sequence
        \begin{align*}
            \omega_j:= e^{-\alpha \rE_m(X^j;\, Y)},
            \qquad V_j:= X^j.
        \end{align*}
        We set $q:=\frac{p(r+2)}{r-p}>p$ so that $r=\frac{p(q+2)}{q-p}$. Then, we define $s:=\frac{pq}{q-p}<r$ and establish the variable $\theta$ in the uniform bounds \eqref{eq:uniform_bnds_for_stationary_sequence} by
        \begin{gather*}
            \int_{\RR^d} e^{-\alpha q \rE_m(x;Y)}\diff \mu(x)
            \leq e^{\alpha q \frac{G}{c}},\qquad
            \int_{\RR^d} |x|^s\,e^{- \alpha s \rE_m(x;Y)}\diff \mu(x)
            \leq e^{\alpha s \frac{G}{c}}\left(\int_{\RR^d} |x|^r\diff \mu(x)\right)^{\frac{s}{r}}.
        \end{gather*}
        By the Marcinkiewicz–Zygmund inequality \cite[Chapter 10.3, Theorem 2]{chow_probability_1997} there exists a constant $B_p>0$ depending on $p$ such that
        \begin{align*}
            \EE[|\hat{N}_J-N|^p]
            &=\frac{1}{J^p}\EE\left[\left\vert \sum_{j=1}^J \omega_j V_j-\EE[\omega_1 V_1]\right\vert^p\right]
            \leq \frac{B_p}{J^p}\EE\left[\left( \sum_{j=1}^J |\omega_j V_j-\EE[\omega_1 V_1]|^2\right)^{\frac{p}{2}}\right].
        \end{align*}
        Next, we use the Jensen inequality to derive
        \begin{align*}
            &\frac{B_p}{J^p}\EE\left[\left( \sum_{j=1}^J |\omega_j V_j-\EE[\omega_1 V_1]|^2\right)^{\frac{p}{2}}\right]
            \leq \frac{B_p}{J^{\frac{p}{2}+1}}\EE\left[ \sum_{j=1}^J |\omega_j V_j-\EE[\omega_1 V_1]|^p\right]\\
            &\qquad=\frac{B_p}{J^{\frac{p}{2}}}\EE\left[ |\omega_1 V_1-\EE[\omega_1 V_1]|^p\right]
            \leq 2^p B_p\EE[ |\omega_1 V_1|^p] J^{-\frac{p}{2}}
            \leq 2^p B_p e^{\alpha p \frac{G}{c}}\left(\int_{\RR^d} |x|^r\diff \mu(x)\right)^{\frac{p}{r}} J^{-\frac{p}{2}}.
        \end{align*}
        We establish a similar bound for $\EE[|\hat{D}_J-D|^q]$. We need to bound the numerator $N$ from above
        \begin{align*}
            N\leq \int_{\RR^d} |x|\,e^{-\alpha \rE_m(x;Y)}\diff \mu(x)
            \leq e^{\alpha \frac{G}{c}}\left(\int_{\RR^d} |x|^r\diff \mu(x)\right)^{\frac{1}{r}},
        \end{align*}
        and the denominator $D$ from below
        \begin{align*}
            D=\int_{\RR^d} e^{-\alpha \rE_m(x;Y)}\diff \mu(x)
            \geq \int_{\RR^d} e^{-\alpha c (|(x,Y)|^\ell+G)}\diff \mu(x)
            \geq \frac{1}{2}e^{-\alpha c ((L+R)^\ell+G)}.
        \end{align*}
        for some $L>0$ such that $\mu(\{x\in\RR^d\;\vert\; |x|\leq L\})\geq \frac{1}{2}$. Thus, we see that the estimates in \eqref{eq:uniform_bnds_for_stationary_sequence} hold uniformly for all $|Y|\leq R$ and that the coefficient in \eqref{eq:mf_stationary_sequence} is uniformly bounded for all $|Y|\leq R$. With this we conclude the proof.
    \end{proof}

    \section{Proof of Theorem \ref{theorem:existence_and_uniqueness_for_cbo}}
    \label{section:wellposedness_for_cbo}

    We use the non-explosion criterion from stochastic Lyapunov theory as stated in \cite[Theorem 3.5]{khasminskii_stochastic_2012}. To that end, we rewrite the stochastic differential equation \eqref{eq:cbo} as
    \begin{align*}
        \diff X_t =F(X_t)\diff t+G(X_t)\diff B_t,
    \end{align*}
    where $X_t:=(X_t^{1,1},\ldots, X_t^{1,N},X_t^{2,1},\ldots,X_t^{M,N})$, $B_t=(B_t^{1,1},\ldots, B_t^{1,N},B_t^{2,1},\ldots,B_t^{M,N})$ and we define the functions
    \begin{align*}
        F(X_t)&:=-\lambda \begin{pmatrix}
            X_t^{1,1}-\con_\alpha^1(\dcbo_t^{1,N},\scbo_t^{-1}) \\ \vdots \\ X_t^{1,N}-\con_\alpha^1(\dcbo_t^{1,N},\scbo_t^{-1}) \\
            X_t^{2,1}-\con_\alpha^2(\dcbo_t^{2,N},\scbo_t^{-2}) \\ \vdots \\
            X_t^{M,N}-\con_\alpha^M(\dcbo_t^{M,N},\scbo_t^{-M})
        \end{pmatrix}\in\RR^{M\cdot N\cdot d},\\ 
        G(X_t)&:=\sigma\begin{pmatrix}
            D(X_t^{1,1}-\con_\alpha^1(\dcbo_t^{1,N},\scbo_t^{-1}))\\
            & \ddots  \\
            && D(X_t^{M,N}-\con_\alpha^M(\dcbo_t^{M,N},\scbo_t^{-M}))
        \end{pmatrix}
        \in\RR^{(M\cdot N\cdot d)\times (M\cdot N\cdot d)}.
    \end{align*}
    We have sublinear growth in the consensus
    \begin{align}
        \label{eq:sub_lin_cons}
        \sum_{m=1}^M |\con_\alpha^m(\dcbo_t^{m,N},\scbo_t^{-m})|
        \leq \sum_{m=1}^M\sum_{i=1}^N \frac{e^{-\alpha \rE_m(X_t^{m,i}; \scbo_t^{-m})}}{\sum_{j=1}^N e^{-\alpha \rE_m(X_t^{m,j}; \scbo_t^{-m})}}\cdot |X_t^{m,i}| 
        =\sum_{m=1}^M\sum_{i=1}^N |X_t^{m,i}|
        \leq \sqrt{MN}|X_t|.
    \end{align}
    Moreover, the functions are locally Lipshitz by Corollary \ref{cor:loc_lip_prop} and the fact that $D$ is Lipshitz continuous. Thus the conditions in \cite[Equation 3.32]{khasminskii_stochastic_2012} are satisfied. It remains to show that the coercive function $\vp(x):=|x|^2$ satisfies the condition \cite[Equation 3.43]{khasminskii_stochastic_2012}, this means showing there exists a constant $C>0$ with
    \begin{align*}
        F(X_t)\cdot \nabla \vp(X_t)+\frac{1}{2} \langle G(X_t),\; \nabla^2\vp(X_t)G(X_t)\rangle
        \leq C \vp(X_t),
    \end{align*}
    where the above inner product denotes the dot product between matrices. The condition immediately follows form the fact that
    \begin{align*}
        F(X_t)\cdot \nabla \vp(X_t)+\frac{1}{2} \inner{G(X_t)}{ \nabla^2\vp(X_t)G(X_t)}
        \leq |F(X_t)|^2+|X_t|^2+ \Vert G(X_t)\Vert^2
        \leq C\vp(X_t),
    \end{align*}
    where the above norm denotes the Frobenius norm, and we used the sublinear growth condition \eqref{eq:sub_lin_cons}. Thus the proof follows from applying \cite[Theorem 3.5]{khasminskii_stochastic_2012}.

    \section{Proof of Theorem \ref{theorem:existence_and_uniqueness_for_mfcbo}}
    \label{section:wellposedness_for_mfcbo}
    We follow the proof of \cite[Theorem 2.4]{gerber_mean-field_2024}, who in turn follow the proof of \cite[Theorem 3.1, Theorem 3.2]{carrillo_analytical_2018}. In this proof we use standard Leray-Schauder fixed point argument. We split the proof up into five steps.

    \smallskip
    
    \underline{\it Solution operator.} For some given $(u^1,\ldots,u^M)\in C^0([0,T],\RR^d)^{\otimes M}$, by using the classical theory of SDEs \cite[Theorem 5.2.1]{bernt_oksendal_stochastic_2003} we can uniquely solve the following SDEs, for $m\in [M]$,
    \begin{align}\label{LCBO}
        \diff Y_t^{m}
        =-\lambda( Y_t^{m}-u_t^m)\diff t
        +\sigma D( Y_t^{m}-u_t^m)\diff B_t^{m},
    \end{align}
    with the initial condition $\law(Y_0^m)=\dmf_0^m$ and Brownian motions $\{B_t^m\}_{m\in [M]}$. We define the laws $\nu_t^m=\mbox{law} (Y_t^m)$. Additionally, the processes $\{Y_t^m\}_{m\in [M]}$ are almost surely continuous.  For any $0\leq r\leq t\leq T$, it holds by Burkholder-Davis-Gundy inequality \cite[Chapter 1, Theorem 7.3]{mao_stochastic_2011} that
    \begin{align}\label{conti}
        \begin{split}
            \EE\left[\sup_{s\in[r,t]}|Y_s^m-Y_r^m|^p\right]  
        \leq &2^{p-1}(t-r)^{p-1}\lambda^p \int_r^t\EE[|Y_s^m-u_s^m|^p]\diff s\\
        &\qquad+C_{BDG}\,2^{p-1}\sigma^p(t-r)^{p/2-1}\int_r^t\EE[|D(Y_s^m-u_s^m)|^p]\diff s.
        \end{split}
    \end{align}
    Letting $r=0$, one obtains
    \begin{align*}
        \EE\left[\sup_{s\in[0,t]}|Y_s^m|^p\right]	\leq C\left(\EE[|Y_0^m|^p]+\|u^m\|_\infty^p+\int_0^t\EE[|Y_s^m|^p]\diff s\right),
    \end{align*}
    where $C>0$ depends only on $T$, $\lambda$, $\sigma$ and $p$. Then it follows from Gronwall's inequality that
    \begin{equation}\label{momentbound}
        \EE\left[\sup_{t\in[0,T]}|Y_t^m|^p\right]\leq C\;\big(\EE[|Y_0^m|^p]+\|u^m\|_\infty^p\big)<\infty.
    \end{equation}
    Therefore, by dominated convergence theorem, we have for any $t\in[0,T]$ that 
    \begin{align*}
        \begin{split}
            \con_\alpha^m(\nu_t^m; \overline{\mathbf{Y}}_t^{-m})
        &=\frac{\int_{\RR^d} y_m\, \omega_\alpha^{\rE_m}(y_m;\overline{\mathbf{Y}}_t^{-m})\diff\nu_t^m(y_m)}{\int_{\RR^d} \omega_\alpha^{\rE_m}(y_m;\overline{\mathbf{Y}}_t^{-m})\diff \nu_t^m(y_m)}\\
     &\qquad\qquad\to \frac{\int_{\RR^d} y_m\, \omega_\alpha^{\rE_m}(y_m;\overline{\mathbf{Y}}_r^{-m})\diff \nu_r^m(y_m)}{\int_{\RR^d} \omega_\alpha^{\rE_m}(y_m;\overline{\mathbf{Y}}_r^{-m})\diff \nu_r^m(y_m)}=	\con_\alpha^m(\nu_r^m; \overline{\mathbf{Y}}_r^{-m})
        \end{split}
        \qquad\text{as } t\to r,
    \end{align*}
    where $\overline{\mathbf{Y}}=\EE\left[\left(Y^1,\ldots,Y^M\right)\right]$. This implies that the function $t\mapsto \con_\alpha^m(\nu_t^m; \overline{\mathbf{Y}}_t^{-m})$ is in $C^0([0,T],\RR^d)$. Thus the following solution operator is well-defined
    \begin{align*}
        \mc{T}:
     \left\{
     \begin{array}{rcl}
          C^0([0,T],\RR^d)^{\otimes M} &\to& C^0([0,T],\RR^d)^{\otimes M}  \\
          (u^1,\ldots,u^M)&\mapsto &(\con_\alpha^1(\nu^1; \overline{\mathbf{Y}}^{-1}),\ldots,\con_\alpha^M(\nu^M; \overline{\mathbf{Y}}^{-M}) )
     \end{array}\right.
    \end{align*}

    \underline{\it $\mc{T}$ is continuous.} Take any two functions $(u^1,\ldots,u^M),(v^1,\ldots,v^M)\in C^0([0,T];\RR^d)^{\otimes M}$ and let the processes $\{Y_t^m\}_{m\in[M]}$, $\{Z_t^m\}_{m\in[M]}$ be the corresponding solutions with respective laws $\{\nu_t^m\}_{m\in[M]}$, $\{\mu_t^m\}_{m\in[M]}$. Define the strategies $\overline{\textbf{Y}}:=\EE[(Y^1,\ldots,Y^M)]$ and $\overline{\textbf{Z}}:=\EE[(Z^1,\ldots,Z^M)]$. We use lemma \ref{lem:wasserstein_stability_estimate_for_weighted_mean} to estimate the difference between the two consensuses
    \begin{align*}
        |\con_\alpha^m(\nu_t^m, \overline{\textbf{Y}}_t^{-m})-\con_\alpha^m(\mu_t^m,\overline{\textbf{Z}}_t^{-m})|^p
        \leq C M^{p-1} \sum_{j=1}^M W_p^p(\nu_t^j,\mu_t^j)
        \leq C M^{p-1} \sum_{j=1}^M \EE\left[\sup_{s\in[0,t]}|Y_s^m-Z_s^m|^p\right].
    \end{align*}
    Then, it holds by Burkholder-Davis-Gundy inequality \cite[Chapter 1, Theorem 7.3]{mao_stochastic_2011} that
    \begin{align*}
            \EE\left[\sup_{s\in[0,t]}|Y_s^m-Z_s^m|^p\right]  
        \leq &2^{p-1}t^{p-1}\lambda^p \int_0^t\EE[|Y_s^m-Z_s^m-u_s^m+v_s^m|^p]\diff s\\
        &\qquad+C_{BDG}\,2^{p-1}\sigma^p(t-r)^{p/2-1}\int_r^t\EE[|(D(Y_s^m-u_s^m)-D(Z_s^m-v_s^m)|^p]\diff s\\
        & \leq C\left(\int_0^t \EE[\;|Y_s^m-Z_s^m|^p\;]\diff s+\int_0^t \EE[\;|u_s^m-v_s^m|^p\;]\diff r\right).
    \end{align*}
    Now by the Gronwall inequality we derive the inequality
    \begin{align*}
        \EE\left[\sup_{s\in[0,T]}|Y_s^m-Z_s^m|^p\right]
        \leq C\, \norm{u^m-v^m}_{L_t^\infty}^p.
    \end{align*}
    Combining all the above inequalities proves that $\mathcal{T}$ is a continuous operator.

    \smallskip
    
    \underline{\it $\mc{T}$ is compact.} To show that $\mc{T}$ is a compact operator, we fix $R>0$ and consider the ball
    \begin{equation*}
        \Lambda_R:=\left\{ (u^1,\ldots,u^M) \in C^0([0,T],\RR^d)^{\otimes M}: \sup_{m\in[M]}\|u^m\|_{\infty}\leq R \right\}\,.
    \end{equation*}
    By the Arzel\`{a}-Ascoli theorem, we have a compact embedding $C^{0,1/2}([0,T],\RR^d)\hookrightarrow C^0([0,T],\RR^d)$. Thus, it suffices to show that $\mc T(\Lambda_R)$ is bounded in $C^{0,1/2}([0,T],\RR^d)^{\otimes M}$. Now we take any function $(u^1,\ldots,u^M) \in \Lambda_R$ and consider the corresponding solution $(Y_t^1,\ldots,Y_t^M)$ of \eqref{LCBO} with pointwise law $(\nu_t^1,\ldots,\nu_t^M)$. We first observe that for any $m\in[M]$ and $p>0$
    \begin{align*}
        \|\con_\alpha^m(\nu^m; \overline{\mathbf{Y}}^{-m})\|_{L_t^\infty}
        &=\sup_{t\in[0,T]}\left|\frac{\int_{\RR^d} y_m\, \omega_\alpha^{\rE_m}(y_m;\overline{\mathbf{Y}}_t^{-m})\diff\nu_t^m(y_m)}{\int_{\RR^d} \omega_\alpha^{\rE_m}(y_m;\overline{\mathbf{Y}}_t^{-m})\diff \nu_t^m(y_m)}\right|\\
        &\leq \sup_{t\in[0,T]}C\;\left(\sum_{j=1}^M\int_{\RR^d}|y_j|^p\diff\nu_t^j(y_j)\right)^{\frac{1}{p}}\,,
    \end{align*}
   where we have used Lemma \ref{lemma:est_mean_cost} in the inequality.  It follows from \eqref{momentbound} that
    \begin{equation*}
            \|\con_\alpha^m(\nu^m; \overline{\mathbf{Y}}^{-m})\|_\infty<\infty\,.
    \end{equation*}
    Furthermore, in view of \eqref{conti} and \eqref{momentbound}, there is a constant $L>0$ depending on $R$, $T$, $\lambda$, $\sigma$ and $p$ such that we have the Hölder continuity
    \begin{equation}
        \EE[|Y_t^m-Y_r^m|^p]\leq L|t-r|^{p/2}\,,
    \end{equation}
    which implies that $W_p(\nu_t^m,\nu_r^m)\leq L^{1/p}|t-r|^{1/2}$. Then by Lemma \ref{lem:wasserstein_stability_estimate_for_weighted_mean}, it follows that the function $t\mapsto \con_\alpha^m(\nu_t^m; \overline{\mathbf{Y}}_t^{-m})$ is H\"{o}lder continuous with exponent $1/2$ for any $m\in [M]$. This completes the proof that $\mc T$ is compact.

    \smallskip
    
    \underline{\it Leray-Schauder fixed point Theorem.} The goal of this step is to apply the Leray-Schauder fixed point Theorem \cite[Theorem 11.3]{gilbarg_elliptic_2001}. To that end, we need to prove that the following set is bounded:
    \begin{equation}\label{set}
        \left\{ (u^1,\ldots,u^M) \in C^0([0,T],\RR^d)^{\otimes M}:~\exists \xi\in [0,1] \mbox{ such that } (u^1,\ldots,u^M) =\xi \mc T(u^1,\ldots,u^M)\right\}\,.
    \end{equation}
    To this end, let $(u^1,\ldots,u^M) \in C^0([0,T],\RR^d)^{\otimes M}$ be such that
    \begin{equation}\label{4.7}
        (u^1,\ldots,u^M) =\xi \mc T(u^1,\ldots,u^M)
    \end{equation}
    for some $\xi\in [0,1]$, and let $(Y^1,\ldots,Y^M)$ denote corresponding solution to \eqref{LCBO}. By \eqref{4.7}, the processes $(Y^1,\ldots,Y^M)$ are also solutions to
          \begin{align}
        \diff Y_t^{m}
        =-\lambda( Y_t^{m}-\xi \con_\alpha^m(\nu_t^m; \overline{\mathbf{Y}}_t^{-m}))\diff t
        +\sigma D( Y_t^{m}-\xi \con_\alpha^m(\nu_t^m; \overline{\mathbf{Y}}_t^{-m}))\diff B_t^{m},\quad \nu_t^m=\mbox{law} (Y_t^m).
    \end{align}
    Then by Lemma \ref{lemma:moment_estimates_for_the_empirical_measures} and Remark \ref{remark:add_xi_to_cbo} we find that
    \begin{equation*}
        \EE\left[\sup_{t\in[0,T]}|Y_t^m|^p\right]\leq \kappa	\cdot\sum_{m=1}^M \EE\left[|Y_0^m|^p\right],
    \end{equation*}
    where $\kappa>0$ is a constant. Applying Lemma \ref{lemma:est_mean_cost}, it follows that $\con_\alpha^m(\nu_t^m; \overline{\mathbf{Y}}_t^{-m})$ can be uniformly bounded in $[0,T]$ in terms of $\sum_{m=1}^M\EE\left[|Y_0^m|^p\right]$. This implies that the set \eqref{set} is indeed bounded. Therefore, we obtain the existence of a fixed point of $\mc T$, which is a solution to \eqref{eq:mf_cbo}, and we establish the moment bounds in \eqref{eq:mbounds} as a byproduct.

    \smallskip
    
    \underline{\it Uniqueness of solution.} Suppose that $(Y^1,\ldots,Y^M)$, $(Z^1,\ldots,Z^M):\Omega\ra  C^0([0,T],\RR^d)^{\otimes M}$ are two fixed points of the map $\mc{T}$. Define the difference $W_t^m:=Y_t^m-Z_t^m$ and let $\nu_t^m:=\law(Y_t^m)$ and $\mu_t^m:=\law(Z_t^m)$. By the Burkholder-Davis-Gundy inequality \cite[Chapter 1, Theorem 7.3]{mao_stochastic_2011} for all $t\in [0,T]$ we have 
    \begin{align*}
        \frac{1}{2^{p-1}}&\EE\left[\sup_{r\in [0,t]}\; |W_r^m|^p\right]
        \leq T^{p-1}\lambda^p\int_0^t \EE[\;|W_r^m-\con_\alpha^m(\nu_r^m; \overline{\mathbf{Y}}_r^{-m})+\con_\alpha^m(\mu_r^m; \overline{\mathbf{Z}}_r^{-m})|^p\;]\diff r\\
        &\quad + C_{\mathrm{BDG}}\, T^{p/2-1}\sigma^p\int_0^t \EE[\;|D(Y_r^m-\con_\alpha^m(\nu_r^m; \overline{\mathbf{Y}}_r^{-m}))-D(Z_r^m-\con_\alpha^m(\mu_r^m; \overline{\mathbf{Z}}_r^{-m})|^p\;]\diff r.\\
        &\leq C\left(\int_0^t \EE[\;|W_r^m|^p\;]\diff r+\int_0^t \EE[\;|\con_\alpha^m(\nu_r^m; \overline{\mathbf{Y}}_r^{-m})-\con_\alpha^m(\mu_r^m; \overline{\mathbf{Z}}_r^{-m})|^p\;]\diff r\right),
    \end{align*}
    for some constant $C>0$ depending on $T$, $\lambda$, $\sigma$ and $p$. By Lemma \ref{lem:wasserstein_stability_estimate_for_weighted_mean} we have that 
    \begin{align*}
        |\con_\alpha^m(\nu_t^m; \overline{\mathbf{Y}}_t^{-m})-\con_\alpha^m(\mu_t^m; \overline{\mathbf{Z}}_t^{-m})|\leq C \sum_{j=1}^M W_p(\nu_t^j,\mu_t^j).
    \end{align*}
    Also by definition of the Wasserstein distance $W_p(\nu_t^m,\mu_t^m)^p\leq \EE[\;|W_t^m|^p\;]$. Thus, we conclude by summing over $m$ 
    \begin{align*}
        \EE\left[\sup_{r\in [0,t]}\; \sum_{m=1}^M |W_r^m|^p\right]
        \leq C\int_0^t \EE\left[\;\sum_{m=1}^M |W_r^m|^p\;\right]\diff r,
    \end{align*}
    for some constant $C>0$ depending on $T$, $\lambda$, $\sigma$, $p$ and $M$. Thus, by the Gronwall inequality and the fact that $\EE[W_0^m]=0$ we get
    \begin{align*}
        \EE\left[\sup_{t\in [0,T]}\; |W_t^m|^p\right]
        =0.
    \end{align*}
    With this we have proven the uniqueness of our solution. Thus, we conclude the proof.

    \section{Proof of Theorem \ref{theorem:meanfield_of_cbo}}
    \label{section:meanfield_limit_for_cbo}
    
    We will follow the same proof outline as \cite[Theorem 2.6]{gerber_mean-field_2024} or \cite[Theorem 3.1]{chaintron_propagation_nodate}.
    \begin{proof}[Proof of Theorem \ref{theorem:meanfield_of_cbo}]

    Let us explain the proof strategy: The goal of the proof is to devise a Gronwall type estimate between the expected particle difference $|\bX_t^{m,i}-X_t^{m,i}|^p$. For this we make use of the SDE's \eqref{eq:cbo} and \eqref{eq:mf_cbo} and bound the difference in consensus between $\dcbo^{m,N}$ and $\overline{\mu}^{m,N}$, and $\overline{\mu}^{m,N}$ and $\dmf^m$, where $\overline{\mu}^{m,N}$ the empirical distribution of \eqref{eq:mf_cbo}. We divide the proof into two steps.

        \smallskip

        \underline{\it Reduction to higher exponents.} First we show that it is sufficient to prove the statement for the case $p\in [2\vee \pM,\, \frac{q}{2}]$. Indeed, if the statement holds true in this case, then for all $r\in (0,2\vee \pM)$ we have by Jensen's inequality
        \begin{align*}
            \left(\EE\left[\sup_{t\in [0,T]}\;\abs{X_t^{m,i}- \bX_t^{m,i}}^r\right]\right)^{\frac{1}{r}}
            \leq \left(\EE\left[\sup_{t\in [0,T]}\;\abs{X_t^{m,i}- \bX_t^{m,i}}^{2\vee \pM}\right]\right)^{\frac{1}{2\vee \pM}}
            \leq C N^{-\beta}
        \end{align*}
        with
        \begin{align*}
            \beta:=\min\left\{\frac{1}{2},\; \frac{q-(2\vee \pM)}{2 (2\vee \pM)^2}\right\}
            =\min\left\{\frac{1}{2},\; \frac{q-r}{2r^2},\; \frac{q-(2\vee \pM)}{2 (2\vee \pM)^2}\right\},
        \end{align*}
        where we used $(0,q)\ni z\mapsto \frac{q-z}{z^2}\in \RR$ is decreasing.

        \smallskip
    
        \underline{\it Mean-field limit.} Let $p\in [2\vee \pM,\, \frac{q}{2}]$ be fixed. We define the particle distribution of \eqref{eq:mf_cbo} as
        \begin{align*}
            \dpmf^{m,N}:=\frac{1}{N}\sum_{i=1}^N \delta_{\bX^{m,i}}.
        \end{align*}
        Then by the Burkholder-Davis-Gundy inequality \cite[Chapter 1, Theorem 7.3]{mao_stochastic_2011} we derive the estimate
        \begin{align*}
            &\EE\left[\sup_{t\in [0,T]}\abs{X_t^{m,i}-\bX_t^{m,i}}^p\right]\\
            &\qquad\leq (2T)^{p-1}\lambda^p\int_0^T \EE\left[\abs{X_t^{m,i}-\bX_t^{m,i}+\con_\alpha^m(\dcbo_t^{m,N},\scbo_t^{-m})-\con_\alpha^m(\dmf_t^m, \smf_t^{-m})}^p\right]\,dt\\
            &\qquad\qquad+2^{p-1} T^{\frac{p}{2}-1}\sigma^p C_{\mathrm{DBG}} \int_0^T \EE\left[\abs{D(X_t^{m,i}-\con_\alpha^m(\dcbo_t^{m,N},\scbo_t^{-m}))-D(\bX_t^{m,i}-\con_\alpha^m(\dmf_t^m,\smf_t^{-m}))}^p\right]\,dt\\
            &\qquad\tag{P Diff}\label{eq:mean-field limit:particle diff}
            \leq (6T)^{p-1}\lambda^p\int_0^T \EE\left[\abs{X_t^{m,i}-\bX_t^{m,i}}^p\right]\,dt\\
            &\qquad\qquad \tag{Emp App}\label{eq:mean-field limit:empirical_approximation}
            + (6T)^{p-1}\lambda^p\int_0^T \EE\left[\abs{\con_\alpha^m(\dmf_t^m,\smf_t^{-m})-\con_\alpha^m(\dpmf_t^{m,N}, \smf_t^{-m})}^p\right]\,dt\\
            &\qquad\qquad \tag{Emp Diff}\label{eq:mean-field limit:empirical_difference}
            + (6T)^{p-1}\lambda^p\int_0^T \EE\left[\abs{\con_\alpha^m(\dpmf_t^{m,N}, \smf_t^{-m})-\con_\alpha^m(\dcbo_t^{m,N}, \scbo_t^{-m})}^p\right]\,dt\\
            &\qquad\qquad \tag{L}\label{eq:mean-field limit:lipshitz transform}
                \begin{aligned}
                    +2^{p-1} T^{\frac{p}{2}-1}\sigma^p C_{\mathrm{DBG}} \int_0^T \EE\Big[\big\vert &D(X_t^{m,i}-\con_\alpha^m(\dcbo_t^{m,N},\scbo_t^{-m}))\\
                    &\qquad-D(\bX_t^{m,i}-\con_\alpha^m(\dmf_t^m, \smf_t^{-m}))\big\vert^p\Big]\,dt.
                \end{aligned}
        \end{align*}
        The next step is to bound each term in the above inequality. First, we start with the empirical difference \eqref{eq:mean-field limit:empirical_difference}. For this we define the $\RR$-valued random variables and value
        \begin{align*}
        	\overline{Z}^{m,i}
        	:=\sup_{t\in[0,T]} |\bX_t^{m,i}|^p,\qquad
        	R_0:=\sup_{m\in [M]}\EE\left[
          \overline{Z}^{m,1}\right].
        \end{align*}
        Let $R>R_0$ be fixed, then we need to define the particle excursion set
        \begin{align*}
        	\Omega_t:=\bigcup_{m\in [M]}\left\{\omega\in\Omega\;\left\vert\;\frac{1}{N}\sum_{i=1}^N |\bX_t^{m,i}(\omega)|^p\geq R\right.\right\}.
        \end{align*}
        Additionally, by Theorem \ref{theorem:existence_and_uniqueness_for_mfcbo} we know that $\EE[|\overline{Z}^{m,i}|^{q/p}]<\infty$. As $q\geq 2p$, we have by \cite[Lemma 2.5]{gerber_mean-field_2024} that there exists a constant $C>0$ independent of $N$ such that 
        \begin{align}
            \label{eq:prob_of_excursion}
            \PP(\Omega_t)\leq C \, N^{-\frac{q}{2p}}.
        \end{align}
        Now we can split the term in \eqref{eq:mean-field limit:empirical_difference} by
        \begin{align*}
        	\EE\left[\abs{\con_\alpha^m(\dpmf_t^{m,N}, \smf_t^{-m})-\con_\alpha^m(\dcbo_t^{m,N}, \scbo_t^{-m})}^p\right]
        	&=\EE\left[\abs{\con_\alpha^m(\dpmf_t^{m,N}, \smf_t^{-m})-\con_\alpha^m(\dcbo_t^{m,N}, \scbo_t^{-m})}^p \bbone_{\Omega\setminus\Omega_t}\right]\\
        	&\qquad+\EE\left[\abs{\con_\alpha^m(\dpmf_t^{m,N}, \smf_t^{-m})-\con_\alpha^m(\dcbo_t^{m,N}, \scbo_t^{-m})}^p \bbone_{\Omega_t}\right].
        \end{align*}
        By Lemma \ref{lem:wasserstein_stability_estimate_for_weighted_mean}, we obtain
        \begin{align*}
        	\EE\left[\abs{\con_\alpha^m(\dpmf_t^{m,N}, \smf_t^{-m})-\con_\alpha^m(\dcbo_t^{m,N}, \scbo_t^{-m})}^p \bbone_{\Omega\setminus\Omega_t}\right]
        	\leq C\sum_{m=1}^M \EE\left[W_p^p(\dpmf_t^{m,N},\dcbo_t^{m,N})\right].
        \end{align*}
        Additionally, the Wasserstein difference is controlled by the particle difference by
        \begin{align*}
        	\EE\left[W_p^p(\dpmf_t^{m,N},\dcbo_t^{m,N})\right]
        	\leq \EE\left[\frac{1}{N}\sum_{i=1}^N |X_t^{m,i}-\bX_t^{m,i}|^p\right]
        	=\EE\left[|X_t^{m,1}-\bX_t^{m,1}|^p\right].
        \end{align*}
        Next, we bound the following consensus using Lemma \ref{lemma:est_mean_cost}, Remark \ref{rmk} and Theorem \ref{theorem:existence_and_uniqueness_for_mfcbo} 
        \begin{align*}
            \EE[\;|\con_\alpha^m(\dpmf_t^{m,N}, \smf_t^{-m})|^q\;]
            \leq C\EE\left[\sum_{j=1}^M\int_{\RR^d}|x|^q\diff \dpmf_t^{j,N}(x)\right]
            \leq C\sum_{j=1}^M \EE\left[\sup_{s\in [0,T]}|\bX_s^{j,1}|^q\right]
            <\infty.
        \end{align*}
        Next, we use the moment estimate from Lemma \ref{lemma:moment_estimates_for_the_empirical_measures} and the previous bound to estimate
        \begin{align*}
        	\EE[|\con_\alpha^m(\dpmf_t^{m,N}, \smf_t^{-m})-\con_\alpha^m(\dcbo_t^{m,N}, \scbo_t^{-m})|^q]
        	&\leq 2^{q-1}\big(\EE[|\con_\alpha^m(\dpmf_t^{m,N}, \smf_t^{-m})|^q]+\EE[|\con_\alpha^m(\dcbo_t^{m,N}, \scbo_t^{-m})|^q]\big)\\
        	&\leq 2^q \kappa
        \end{align*}
        for some constant $\kappa>0$ independent of $N$. Then by the Hölder inequality and \eqref{eq:prob_of_excursion}, we find
        \begin{align*}
        	&\EE[|\con_\alpha^m(\dpmf_t^{m,N}, \smf_t^{-m})-\con_\alpha^m(\dcbo_t^{m,N}, \scbo_t^{-m})|^p \bbone_{\Omega_t}]\\
        	&\qquad\leq \EE[|\con_\alpha^m(\dpmf_t^{m,N}, \smf_t^{-m})-\con_\alpha^m(\dcbo_t^{m,N}, \scbo_t^{-m})|^q]^{\frac{p}{q}}
        	\EE[\bbone_{\Omega_t}^{\frac{q}{q-p}}]^{\frac{q-p}{q}}
        	\leq 2^p C\kappa^{\frac{p}{q}}N^{-\frac{q-p}{2p}}.
        \end{align*}
        We have thus shown
        \begin{align*}
        	\EE[|\con_\alpha^m(\dpmf_t^{m,N}, \smf_t^{-m})-\con_\alpha^m(\dcbo_t^{m,N}, \scbo_t^{-m})|^p]
        	\leq CN^{-\frac{q-p}{2p}}+C\sum_{m=1}^M\EE[|X_t^{m,i}-\bX_t^{m,i}|^p].
        \end{align*}
        For the next step, we observe that by Theorem \ref{theorem:existence_and_uniqueness_for_mfcbo} there exists $R>0$ such that, for all $t\in [0,T]$ and $m\in [M]$, 
        \begin{align*}
            |\smf_t^m|
            \leq \EE\left[|\bX_t^{m,1}|^q\right]^{\frac{1}{q}}
            \leq \EE\left[\sup_{s\in [0,T]}|\bX^{m,1}|^q\right]^{\frac{1}{q}}
            \leq R.
        \end{align*}
        Thus, for the empirical approximation \eqref{eq:mean-field limit:empirical_approximation} we use lemma \ref{lemma:convergence_of_the_weighted_mean_for_iid_samples}, with $\mu=\dmf_t^m$ and $Y=\smf_t^{-m}$, to derive the bound
        \begin{align*}
        	\EE\left[\abs{\con_\alpha^m(\dmf_t^m,\smf_t^{-m})-\con_\alpha^m(\dpmf_t^{m,N}, \smf_t^{-m})}^p\right]
        	\leq CN^{-\frac{p}{2}},
        \end{align*}
        for some constant $C>0$ dependent on $M$ and $R$, but independent of $N$. Finally, using the fact that $D$ is globally Lipshitz, we can bound the term \eqref{eq:mean-field limit:lipshitz transform} analogously as we have done above. This means we derive the inequality 
        \begin{align*}
        	\EE\left[\sup_{t\in [0,T]}\abs{X_t^{m,i}-\bX_t^{m,i}}^p\right]
        	\leq CN^{-\theta p}+C\int_0^T \sum_{j=1}^M \EE\left[\sup_{s\in [0,t]}\abs{X_s^{j,i}-\bX_s^{j,i}}^p\right]\diff t
        \end{align*}
        and we define the parameter
        \begin{align*}
            \theta:=\min\left\{\frac{1}{2},\; \frac{q-p}{2p^2}\right\}
            =\min\left\{\frac{p}{2},\; \frac{q-p}{2p},\; \frac{q-(2\vee \pM)}{2 (2\vee \pM)^2}\right\},
        \end{align*}
        where we used $(0,q)\ni z\mapsto \frac{q-z}{z^2}\in \RR$ is decreasing. We take the sum over $m=1,\ldots,M$ and derive the inequality
        \begin{align*}
        	\sum_{m=1}^M\EE\left[\sup_{t\in [0,T]}\abs{X_t^{m,i}-\bX_t^{m,i}}^p\right]
        	\leq CN^{-\theta p}+C\int_0^T \sum_{m=1}^M \EE\left[\sup_{s\in [0,t]}\abs{X_s^{m,i}-\bX_s^{m,i}}^p\right]\diff t.
        \end{align*}
        The result is now a consequence of the Grönwall inequality and taking the $p$th root. With this we conclude the proof.
    \end{proof}

\section*{Acknowledgements}
HH acknowledges the support of the IMSC Young Investigator Fellowship from NAWI Graz, which has enabled his visit to the University of Oxford in 2024. JW was supported by the Engineering and Physical Sciences Research Council (grant number EP/W524311/1). HH and JW would like to thank José A. Carrillo from the university of Oxford for the useful discussions.

    \bibliography{references.bib} 

\begin{thebibliography}{10}

\bibitem{bayraktar2025uniform}
Erhan Bayraktar, Ibrahim Ekren, and Hongyi Zhou.
\newblock Uniform-in-time weak propagation of chaos for consensus-based optimization.
\newblock {\em arXiv preprint arXiv:2502.00582}, 2025.

\bibitem{beddrich2024constrained}
Jonas Beddrich, Enis Chenchene, Massimo Fornasier, Hui Huang, and Barbara Wohlmuth.
\newblock Constrained consensus-based optimization and numerical heuristics for the few particle regime.
\newblock {\em arXiv preprint arXiv:2410.10361}, 2024.

\bibitem{berger_browns_2007}
Ulrich Berger.
\newblock Brown's original fictitious play.
\newblock {\em Journal of Economic Theory}, 135(1):572--578, July 2007.

\bibitem{bernt_oksendal_stochastic_2003}
{Bernt Øksendal}.
\newblock {\em Stochastic differential equations: an introduction with applications}.
\newblock Universitext. Springer, 2003.

\bibitem{borghi_adaptive_2023}
Giacomo Borghi, Michael Herty, and Lorenzo Pareschi.
\newblock An {Adaptive} {Consensus} {Based} {Method} for {Multi}-objective {Optimization} with {Uniform} {Pareto} {Front} {Approximation}.
\newblock {\em Applied Mathematics \& Optimization}, 88(2):58, October 2023.

\bibitem{borghi2024particle}
Giacomo Borghi, Hui Huang, and Jinniao Qiu.
\newblock A particle consensus approach to solving nonconvex-nonconcave min-max problems.
\newblock {\em arXiv preprint arXiv:2407.17373}, 2024.

\bibitem{bungert2025mirrorcbo}
Leon Bungert, Franca Hoffmann, Doh~Yeon Kim, and Tim Roith.
\newblock Mirrorcbo: A consensus-based optimization method in the spirit of mirror descent.
\newblock {\em arXiv preprint arXiv:2501.12189}, 2025.

\bibitem{bungert_polarized_2022}
Leon Bungert, Tim Roith, and Philipp Wacker.
\newblock Polarized consensus-based dynamics for optimization and sampling, 2022.
\newblock Version Number: 3.

\bibitem{carrillo_analytical_2018}
José~A. Carrillo, Young-Pil Choi, Claudia Totzeck, and Oliver Tse.
\newblock An analytical framework for consensus-based global optimization method.
\newblock {\em Mathematical Models and Methods in Applied Sciences}, 28(06):1037--1066, June 2018.

\bibitem{carrillo_consensus-based_2021}
José~A. Carrillo, Shi Jin, Lei Li, and Yuhua Zhu.
\newblock A consensus-based global optimization method for high dimensional machine learning problems.
\newblock {\em ESAIM: Control, Optimisation and Calculus of Variations}, 27:S5, 2021.

\bibitem{chaintron_propagation_nodate}
Louis-Pierre Chaintron and Antoine Diez.
\newblock Propagation of chaos: A review of models, methods and applications. ii. applications.
\newblock {\em Kinetic \& Related Models}, 15(6), 2022.

\bibitem{chenchene_consensus-based_2023}
Enis Chenchene, Hui Huang, and Jinniao Qiu.
\newblock A consensus-based algorithm for non-convex multiplayer games.
\newblock {\em Journal of Optimization Theory and Applications}, 2025.
\newblock to appear.

\bibitem{chow_probability_1997}
Yuan~Shih Chow and Henry Teicher.
\newblock {\em Probability {Theory}}.
\newblock Springer {Texts} in {Statistics}. Springer, New York, NY, 1997.
\newblock ISSN: 1431-875X.

\bibitem{dai_sbeed_2018}
Bo~Dai, Albert Shaw, Lihong Li, Lin Xiao, Niao He, Zhen Liu, Jianshu Chen, and Le~Song.
\newblock {SBEED}: {Convergent} reinforcement learning with nonlinear function approximation.
\newblock In {\em International conference on machine learning}, pages 1125--1134. PMLR, 2018.

\bibitem{deng_local_2021}
Yuyang Deng and Mehrdad Mahdavi.
\newblock Local {Stochastic} {Gradient} {Descent} {Ascent}: {Convergence} {Analysis} and {Communication} {Efficiency}.
\newblock In {\em Proceedings of {The} 24th {International} {Conference} on {Artificial} {Intelligence} and {Statistics}}, pages 1387--1395. PMLR, March 2021.
\newblock ISSN: 2640-3498.

\bibitem{doukhan_evaluation_2009}
Paul Doukhan and Gabriel Lang.
\newblock Evaluation for moments of a ratio with application to regression estimation.
\newblock {\em Bernoulli. Official Journal of the Bernoulli Society for Mathematical Statistics and Probability}, 15(4):1259--1286, 2009.

\bibitem{fan_fault-tolerant_2021}
Xiaofeng Fan, Yining Ma, Zhongxiang Dai, Wei Jing, Cheston Tan, and Bryan Kian~Hsiang Low.
\newblock Fault-{Tolerant} {Federated} {Reinforcement} {Learning} with {Theoretical} {Guarantee}.
\newblock In {\em Advances in {Neural} {Information} {Processing} {Systems}}, volume~34, pages 1007--1021. Curran Associates, Inc., 2021.

\bibitem{fornasier_consensus-based_2020}
Massimo Fornasier, Hui Huang, Lorenzo Pareschi, and Philippe Sünnen.
\newblock Consensus-based optimization on hypersurfaces: {Well}-posedness and mean-field limit.
\newblock {\em Mathematical Models and Methods in Applied Sciences}, 30(14):2725--2751, December 2020.

\bibitem{fornasier_consensus-based_2021}
Massimo Fornasier, Hui Huang, Lorenzo Pareschi, and Philippe Sünnen.
\newblock Consensus-based optimization on the sphere: {Convergence} to global minimizers and machine learning.
\newblock {\em Journal of Machine Learning Research}, 22(237):1--55, 2021.

\bibitem{fornasier_consensus-based_2024}
Massimo Fornasier, Timo Klock, and Konstantin Riedl.
\newblock Consensus-{Based} {Optimization} {Methods} {Converge} {Globally}.
\newblock {\em SIAM Journal on Optimization}, 34(3):2973--3004, September 2024.
\newblock arXiv:2103.15130 [math].

\bibitem{gerber2025uniform}
Nicolai Gerber, Franca Hoffmann, Dohyeon Kim, and Urbain Vaes.
\newblock Uniform-in-time propagation of chaos for consensus-based optimization.
\newblock {\em arXiv preprint arXiv:2505.08669}, 2025.

\bibitem{gerber_mean-field_2024}
Nicolai~Jurek Gerber, Franca Hoffmann, and Urbain Vaes.
\newblock Mean-field limits for {Consensus}-{Based} {Optimization} and {Sampling}.
\newblock {\em arXiv:2312.07373v2}, 2024.

\bibitem{gilbarg_elliptic_2001}
David Gilbarg and Neil~S. Trudinger.
\newblock {\em Elliptic partial differential equations of second order}.
\newblock Classics in mathematics. Springer, Berlin, reprint of 1998 ed. edition, 2001.

\bibitem{gokhale_evolutionary_2014}
Chaitanya~S. Gokhale and Arne Traulsen.
\newblock Evolutionary {Multiplayer} {Games}.
\newblock {\em Dynamic Games and Applications}, 4(4):468--488, December 2014.
\newblock arXiv:1404.1421 [q-bio].

\bibitem{ha_stochastic_2022}
Seung-Yeal Ha, Myeongju Kang, Dohyun Kim, Jeongho Kim, and Insoon Yang.
\newblock Stochastic consensus dynamics for nonconvex optimization on the {Stiefel} manifold: {Mean}-field limit and convergence.
\newblock {\em Mathematical Models and Methods in Applied Sciences}, February 2022.
\newblock Publisher: World Scientific Publishing Company.

\bibitem{han_adaptive_2018}
Honggui Han, Wei Lu, Lu~Zhang, and Junfei Qiao.
\newblock Adaptive {Gradient} {Multiobjective} {Particle} {Swarm} {Optimization}.
\newblock {\em IEEE Transactions on Cybernetics}, 48(11):3067--3079, November 2018.

\bibitem{hart_simple_2000}
Sergiu Hart and Andreu Mas-Colell.
\newblock A {Simple} {Adaptive} {Procedure} {Leading} to {Correlated} {Equilibrium}.
\newblock {\em Econometrica}, 68(5):1127--1150, September 2000.

\bibitem{huang2024uniform}
Hui Huang and Hicham Kouhkouh.
\newblock Uniform-in-time mean-field limit estimate for the consensus-based optimization.
\newblock {\em arXiv preprint arXiv:2411.03986}, 2024.

\bibitem{huang2025self}
Hui Huang and Hicham Kouhkouh.
\newblock Self-interacting cbo: Existence, uniqueness, and long-time convergence.
\newblock {\em Applied Mathematics Letters}, 161:109372, 2025.

\bibitem{huang_mean-field_2022}
Hui Huang and Jinniao Qiu.
\newblock On the mean-field limit for the consensus-based optimization.
\newblock {\em Mathematical Methods in the Applied Sciences}, 45(12):7814--7831, August 2022.
\newblock arXiv:2105.12919 [math].

\bibitem{huang2023global}
Hui Huang, Jinniao Qiu, and Konstantin Riedl.
\newblock On the global convergence of particle swarm optimization methods.
\newblock {\em Applied Mathematics \& Optimization}, 88(2):30, 2023.

\bibitem{huang2024consensus}
Hui Huang, Jinniao Qiu, and Konstantin Riedl.
\newblock Consensus-based optimization for saddle point problems.
\newblock {\em SIAM Journal on Control and Optimization}, 62(2):1093--1121, 2024.

\bibitem{jingrun_chen_consensus-based_2022}
Jingrun~Chen Jingrun~Chen, Shi~Jin Shi~Jin, and Liyao~Lyu Liyao~Lyu.
\newblock A {Consensus}-{Based} {Global} {Optimization} {Method} with {Adaptive} {Momentum} {Estimation}.
\newblock {\em Communications in Computational Physics}, 31(4):1296--1316, January 2022.

\bibitem{kalise_consensus-based_2023}
Dante Kalise, Akash Sharma, and Michael~V. Tretyakov.
\newblock Consensus-based optimization via jump-diffusion stochastic differential equations.
\newblock {\em Mathematical Models and Methods in Applied Sciences}, 33(02):289--339, February 2023.

\bibitem{khasminskii_stochastic_2012}
Rafail Khasminskii.
\newblock {\em Stochastic {Stability} of {Differential} {Equations}}, volume~66 of {\em Stochastic {Modelling} and {Applied} {Probability}}.
\newblock Springer, Berlin, Heidelberg, 2012.

\bibitem{king-casas_understanding_2012}
Brooks King-Casas and Pearl~H. Chiu.
\newblock Understanding {Interpersonal} {Function} in {Psychiatric} {Illness} {Through} {Multiplayer} {Economic} {Games}.
\newblock {\em Biological psychiatry}, 72(2):119--125, July 2012.

\bibitem{lanctot_unified_2017}
Marc Lanctot, Vinicius Zambaldi, Audrunas Gruslys, Angeliki Lazaridou, Karl Tuyls, Julien Pérolat, David Silver, and Thore Graepel.
\newblock A unified game-theoretic approach to multiagent reinforcement learning.
\newblock {\em Advances in Neural Information Processing Systems}, 30, 2017.

\bibitem{li_triple_2017}
Chongxuan LI, Taufik Xu, Jun Zhu, and Bo~Zhang.
\newblock Triple {Generative} {Adversarial} {Nets}.
\newblock In {\em Advances in {Neural} {Information} {Processing} {Systems}}, volume~30. Curran Associates, Inc., 2017.

\bibitem{maehara_budget_2015}
Takanori Maehara, Akihiro Yabe, and Ken-ichi Kawarabayashi.
\newblock Budget {Allocation} {Problem} with {Multiple} {Advertisers}: {A} {Game} {Theoretic} {View}.
\newblock In {\em Proceedings of the 32nd {International} {Conference} on {Machine} {Learning}}, pages 428--437. PMLR, June 2015.
\newblock ISSN: 1938-7228.

\bibitem{mao_stochastic_2011}
Xuerong Mao.
\newblock {\em Stochastic differential equations and applications}.
\newblock Woodhead Publishing, second edition, reprinted edition, 2011.

\bibitem{narahari_game_2014}
Y~Narahari.
\newblock {\em Game {Theory} and {Mechanism} {Design}}, volume~4 of {\em {IISc} {Lecture} {Notes} {Series}}.
\newblock WORLD SCIENTIFIC / INDIAN INST OF SCIENCE, INDIA, May 2014.

\bibitem{nash_equilibrium_1950}
John~F. Nash.
\newblock Equilibrium points in \textit{n} -person games.
\newblock {\em Proceedings of the National Academy of Sciences}, 36(1):48--49, January 1950.

\bibitem{nouiehed_solving_2019}
Maher Nouiehed, Maziar Sanjabi, Tianjian Huang, Jason~D Lee, and Meisam Razaviyayn.
\newblock Solving a {Class} of {Non}-{Convex} {Min}-{Max} {Games} {Using} {Iterative} {First} {Order} {Methods}.
\newblock In {\em Advances in {Neural} {Information} {Processing} {Systems}}, volume~32. Curran Associates, Inc., 2019.

\bibitem{pinnau_consensus-based_2017}
René Pinnau, Claudia Totzeck, Oliver Tse, and Stephan Martin.
\newblock A consensus-based model for global optimization and its mean-field limit.
\newblock {\em Mathematical Models and Methods in Applied Sciences}, 27(01):183--204, January 2017.

\bibitem{sedarous_multi-swarm_2018}
Shery Sedarous, Sherin~M. El-Gokhy, and Elsayed Sallam.
\newblock Multi-swarm multi-objective optimization based on a hybrid strategy.
\newblock {\em Alexandria Engineering Journal}, 57(3):1619--1629, September 2018.

\bibitem{song_multi-agent_2018}
Jiaming Song, Hongyu Ren, Dorsa Sadigh, and Stefano Ermon.
\newblock Multi-{Agent} {Generative} {Adversarial} {Imitation} {Learning}.
\newblock In {\em Advances in {Neural} {Information} {Processing} {Systems}}, volume~31. Curran Associates, Inc., 2018.

\bibitem{totzeck2021trends}
Claudia Totzeck.
\newblock Trends in consensus-based optimization.
\newblock In {\em Active Particles, Volume 3: Advances in Theory, Models, and Applications}, pages 201--226. Springer, 2021.

\bibitem{trillos2024cb}
Nicol{\'a}s~Garc{\'\i}a Trillos, Sixu Li, Konstantin Riedl, and Yuhua Zhu.
\newblock {CB$^2$O}: Consensus-based bi-level optimization.
\newblock {\em arXiv preprint arXiv:2411.13394}, 2024.

\bibitem{wang_particle_2018}
Dongshu Wang, Dapei Tan, and Lei Liu.
\newblock Particle swarm optimization algorithm: an overview.
\newblock {\em Soft Computing}, 22(2):387--408, January 2018.

\bibitem{wang2025mathematical}
Jinhuan Wang, Keyu Li, and Hui Huang.
\newblock Mathematical analysis of the pde model for the consensus-based optimization.
\newblock {\em arXiv preprint arXiv:2504.10990}, 2025.

\end{thebibliography}
    \bibliographystyle{plain}
\end{document}